\documentclass[leqno,11pt]{amsart}

\usepackage{amsmath,amstext,amssymb,amsopn,amsthm,mathrsfs}

\title[Sobolev spaces for Jacobi expansions]
{Sobolev spaces \\ associated with Jacobi expansions}

\author[B{.} Langowski]{Bartosz Langowski}

\address{Bartosz Langowski \newline
			Institute of Mathematics and Computer Science \newline
      Wroc\l{}aw University of Technology       \newline
      Wyb{.} Wyspia\'nskiego 27,
      50--370 Wroc\l{}aw, Poland      
      }
\email{bartosz.langowski@pwr.wroc.pl}

\allowdisplaybreaks

\theoremstyle{plain}
\newtheorem{thm}{Theorem}[section]
\newtheorem{bigthm}{Theorem}
\newtheorem{lem}[thm]{Lemma}
\newtheorem{prop}[thm]{Proposition}
\newtheorem{cor}[thm]{Corollary}

\theoremstyle{definition}

\theoremstyle{remark}
\newtheorem*{rem*}{Remark}

\theoremstyle{plain}

\DeclareMathOperator{\spann}{span}

\DeclareMathOperator{\support}{supp}

\DeclareMathOperator{\id}{Id}

\def\Z{\mathbb Z}

\def\Lp{\mathcal L_{\ab}^{p,m}}

\def\P{\mathcal P}

\def\ab{\alpha,\beta}

\def\ph{\phi_n^{\ab}}

\def\D{\mathcal D}

\def\gr{\lim_{l\rightarrow\infty}}

\begin{document}

\begin{abstract}
We define and study Sobolev spaces associated with Jacobi expansions. We prove that
these Sobolev spaces are isomorphic to Jacobi potential spaces. As a technical tool, we also show some
approximation properties of Poisson-Jacobi integrals.
\end{abstract}

\maketitle

\footnotetext{
\emph{\noindent Mathematics Subject Classification:} primary 42C10; secondary 42C05, 42C20.\\
\emph{Key words and phrases:} Jacobi expansion, Sobolev space, potential space,
	Poisson-Jacobi integral, maximal operator.
} 

\section{Introduction} \label{sec:intro}

Sobolev spaces associated to Hermite and Laguerre expansions were investigated not long ago in
\cite{BT1,BT2,Graczyk,RT}. Recently Betancor et al.\ \cite{betancor} studied Sobolev spaces in the
context of ultraspherical expansions. Inspired by \cite{betancor}, in this paper we define and study
Sobolev spaces in a more general situation of Jacobi expansions. Noteworthy, analysis related to Jacobi
expansions received a considerable attention over the last fifty years. For the corresponding developments
in the recent years, see for instance
\cite{BaUr1,BaUr2,CU,CNS,L,LZ,LSL,MT,Nowak&Roncal,NoSj,NS1,NS2,parameters,Stempak}.

Our motivation is, first of all, to extend definitions and results from \cite{betancor} to the framework
of Jacobi expansions. Another motivation comes from a question of removing the restriction on the 
ultraspherical parameter of type $\lambda$ imposed throughout \cite{betancor}. In this paper we admit
all possible Jacobi parameters of type $\alpha,\beta$, thus also all possible $\lambda$.
Finally, still another motivation originates in the very definition of the ultraspherical Sobolev spaces
proposed in \cite{betancor}. It is based on higher order `derivatives' involving first order differential
operators related to various parameters of type $\lambda$. Here we consider another, seemingly more natural,
definition of Jacobi Sobolev spaces by means of higher order `derivatives' linked to one fixed pair of
the type parameters $\alpha,\beta$. The concept of higher order `derivative' we employ was postulated
recently by Nowak and Stempak \cite{symmetrized}, 
and the question of its relevance to the theory of Sobolev spaces was posed there. 
Perhaps a bit unexpectedly, we show that the associated Sobolev spaces are not quite appropriate.

Given parameters $\alpha, \beta>-1$, consider the Jacobi differential operator
\begin{equation*} 
L_{\ab}=-\frac{d^2}{d\theta^2} -\frac{1-4\alpha^2}{16\sin^2\frac{\theta}{2}}-\frac{1-4\beta^2}{16\cos^2\frac{\theta}{2}}
 = D_{\ab}^* D_{\ab} + A^2_{\ab}; 
\end{equation*}
here $A_{\ab}=(\alpha+\beta+1)/2$,
$D_{\ab}=\frac{d}{d\theta}-\frac{2\alpha+1}{4}\cot\frac{\theta}{2}+\frac{2\beta+1}{4}\tan\frac{\theta}{2}$
is the first order `derivative' naturally associated to $L_{\ab}$, and 
$D_{\ab}^* = D_{\ab}-2\frac{d}{d\theta}$ is its formal adjoint in $L^2(0,\pi)$.
It is well known that $L_{\ab}$, defined initially on $C_c^2(0,\pi)$, has a non-negative self-adjoint
extension to $L^2(0,\pi)$ whose spectral decomposition is discrete and given by the Jacobi functions
$\phi_n^{\ab}$, $n \ge 0$. The corresponding eigenvalues are $\lambda_n^{\ab} = (n+A_{\ab})^2$,
and the system $\{\phi_n^{\ab}:n\ge 0\}$ constitutes an orthonormal basis in $L^2(0,\pi)$; 
see Section~\ref{sec:tech} for more details. If $\alpha+1/2=\beta+1/2 =: \lambda$, then
the Jacobi context reduces to the ultraspherical situation considered in \cite{betancor}.

When $\ab \ge -1/2$, the functions $\phi_n^{\ab}$ belong to all $L^p(0,\pi)$, $1\le p \le \infty$.
However, if $\alpha< -1/2$ or $\beta < -1/2$, then $\phi_n^{\ab}$ are in $L^p(0,\pi)$ if and only if
$p < p(\ab):=-1/\min(\alpha+1/2,\beta+1/2)$. This leads to the so-called pencil phenomenon 
(cf.\ \cite{NoSj3}) manifesting in the restriction $p'(\ab) < p < p(\ab)$ for $L^p$ mapping
properties of various operators associated with $L_{\ab}$ (here $p'$ denotes the conjugate exponent of $p$,
$1/p+1/p' =1$). Consequently, our main results are restricted to $p \in E(\ab)$, where
$$
E(\ab) := \begin{cases}
						(1,\infty), &  \ab \ge -1/2, \\
						\big(p'(\ab),p(\ab)\big), & \textrm{otherwise}.
					\end{cases}
$$

Let $\sigma>0$. For $\alpha+\beta \neq -1$, consider the potential operator $L_{\ab}^{-\sigma}$.
When $\alpha+\beta = -1$, zero is the eigenvalue of $L_{\ab}$ and hence we consider instead the Bessel
type potential operator $(\id+L_{\ab})^{-\sigma}$. In both cases the potentials are well defined spectrally
and are
bounded on $L^2(0,\pi)$ and possess integral representations valid not only in $L^2(0,\pi)$, 
but also far beyond that space. We will show that $L_{\ab}^{-\sigma}$ and $(\id+L_{\ab})^{-\sigma}$
are one to one and bounded on $L^p(0,\pi)$, $p \in E(\ab)$. Thus, given $s>0$ and $p \in E(\ab)$, 
it makes sense to define the Jacobi potential spaces as the ranges of the potential operators on $L^p(0,\pi)$,
$$
\mathcal{L}_{\ab}^{p,s} := \begin{cases}
															L_{\ab}^{-s/2}\big( L^p(0,\pi)\big), & \alpha+\beta \neq -1,\\
															(\id+L_{\ab})^{-s/2}\big( L^p(0,\pi)\big), & \alpha+\beta = -1.
														\end{cases}
$$
Then the formula
$$
\|f\|_{\mathcal{L}_{\ab}^{p,s}} := \|g \|_{L^p(0,\pi)}, \qquad
																		\begin{cases} 
																			 f=L^{-s/2}_{\ab}g, \quad g \in L^p(0,\pi),
																			 		& \alpha+\beta \neq -1,\\														
																			 f = (\id+L_{\ab})^{-s/2}g, \quad g \in L^p(0,\pi),
																			 		& \alpha+\beta = -1,
																		\end{cases}
$$
defines a norm on $\mathcal{L}_{\ab}^{p,s}$ and it is straightforward to check that $\mathcal{L}_{\ab}^{p,s}$
equipped with this norm is a Banach space.

According to a general concept, Sobolev spaces $\mathbb{W}_{\ab}^{p,m}$, $m \ge 1$, associated to $L_{\ab}$
should be defined by
$$
\mathbb{W}_{\ab}^{p,m} := \big\{ f \in L^p(0,\pi) : \mathbb{D}^{(k)}f \in L^p(0,\pi), k=1,\ldots,m \big\}
$$
and equipped with the norms
$$
\|f\|_{\mathbb{W}_{\ab}^{p,m}} := \sum_{k=0}^{m} \|\mathbb{D}^{(k)}f\|_{L^p(0,\pi)}.
$$
Here $\mathbb{D}^{(k)}$ are suitably defined differential operators of orders $k$ playing the role of
higher order derivatives, and the differentiation is understood in a weak sense. Thus $\mathbb{W}_{\ab}^{p,m}$
depends on a proper choice of $\mathbb{D}^{(k)}$, which is actually the heart of the matter.
It was shown in \cite{betancor} that even in the ultraspherical case seemingly the most natural choice
$\mathbb{D}^{(k)} = D_{\ab}^k$ is not appropriate since then the spaces $\mathbb{W}_{\ab}^{p,m}$ and
$\mathcal{L}_{\ab}^{p,m}$ are not isomorphic in general. 

On the other hand, the isomorphism between Sobolev and potential spaces is a crucial aspect of the
classical theory that should be preserved in the present setting. With this motivation, inspired by
\cite{betancor}, we introduce the higher order `derivative'
$$
D^{(k)} := D_{\alpha+k-1,\beta+k-1} \circ \ldots \circ D_{\alpha+1,\beta+1} \circ D_{\ab}.
$$
Then taking $\mathbb{D}^{(k)}=D^{(k)}$ we get Sobolev spaces satisfying the desired property.
Denote
$$
W_{\ab}^{p,m} := \big\{ f \in L^p(0,\pi) : D^{(k)}f \in L^p(0,\pi), k=1,\ldots,m \big\}.
$$
We will prove the following.
\begin{bigthm} \label{thm:sob1}
Let $\ab>-1$, $p \in E(\ab)$ and $m \ge 1$. Then
$$
W_{\ab}^{p,m} = \mathcal{L}_{\ab}^{p,m}
$$
in the sense of isomorphism of Banach spaces.
\end{bigthm}

Notice that the higher order `derivative' $D^{(k)}$ has a philosophical disadvantage that is the dependence
on the first order `derivatives' related to variable parameters of type. Thus we ask if it is possible to
overcome this inconvenience by introducing still another notion of higher order `derivative'
$$
\mathcal{D}^{(k)} := 
\underbrace{\ldots D_{\ab} D_{\ab}^* D_{\ab} D_{\ab}^* D_{\ab}}_{k\; \textrm{components}}.
$$
This choice resulting from interlacing $D_{\ab}$ with $D_{\ab}^*$ is postulated and supported in 
\cite{symmetrized}, and the question of its utility in the theory of Sobolev spaces was posed in 
\cite[p.\,441]{symmetrized}.
Clearly, $\mathcal{D}^{(k)}$ depends only on $D_{\ab}$ (one pair of type parameters
involved) and has a simpler structure than $D^{(k)}$ because $D_{\ab}^* D_{\ab} = L_{\ab}-A_{\ab}^2$.
Unfortunately, $\mathcal{D}^{(k)}$ turns out to be unsuitable for defining the Sobolev spaces, leading 
in fact to essentially larger Sobolev spaces than $D^{(k)}$. Let
$$
\mathcal{W}_{\ab}^{p,m} := \big\{ f \in L^p(0,\pi) : \mathcal{D}^{(k)}f \in L^p(0,\pi), k=1,\ldots,m \big\}.
$$
\begin{bigthm} \label{thm:sob2}
Let $\ab>-1$, $p \in E(\ab)$ and $m \ge 1$. Then
$$
\mathcal{L}_{\ab}^{p,m} \subset \mathcal{W}_{\ab}^{p,m}
$$
in the sense of embedding of Banach spaces.
However, the reverse inclusion does not hold for all parameters values. In particular,
for each $\alpha,\beta$ satisfying $0 \neq \alpha, \beta < 1/p-1/2$ there is 
$f \in \mathcal{W}_{\ab}^{p,2}$ such that $f \notin \mathcal{L}_{\ab}^{p,2}$.
\end{bigthm}

Theorems \ref{thm:sob1} and \ref{thm:sob2} are the main results of the paper. 
Proving the first one requires actually more technical effort, but we follow
a similar strategy to that in \cite{betancor} aiming at demonstrating that the relevant norms are equivalent. 
Roughly, we shall show estimates of the form
$\|D^{(k)}f\|_{L^p(0,\pi)} \sim \|L_{\ab}^{k/2}f\|_{L^p(0,\pi)}$ or equivalently,
$\|D^{(k)}L_{\ab}^{-k/2}g\|_{L^p(0,\pi)} \sim \|g\|_{L^p(0,\pi)}$. Therefore we will need to
prove essentially two things: first, $L^p$-boundedness of the operators $D^{(k)}L_{\ab}^{-k/2}$
that may be regarded as analogues of the classical higher order Riesz transforms; second, existence of a certain
inversion procedure that will enable us to write bounds of the form 
$\|g\|_{L^p(0,\pi)} \le C \|D^{(k)}L_{\ab}^{-k/2}g\|_{L^p(0,\pi)}$. The latter task will
require introducing some auxiliary operators and studying their $L^p$ mapping properties. The main
technical tool applied repeatedly will be a powerful multiplier-transplantation theorem due to
Muckenhoupt \cite{Muckenhoupt}. The same result was used in \cite{betancor}, but here we apply it in
a slightly simpler way. Another tool we shall need are some approximation properties of Poisson-Jacobi
integrals. In particular, we will obtain certain results of independent interest for the corresponding
maximal operators.

The paper is organized as follows. In Section \ref{sec:tech} we describe in detail the setting and 
give some preparatory results, including the aforementioned multiplier-transplantation theorem and properties
of the Poisson-Jacobi integrals. Sections \ref{sec:bet} and \ref{sec:alternating} are devoted to the
proofs of Theorems \ref{thm:sob1} and \ref{thm:sob2}, respectively. Finally, in Section \ref{sec:fin}
we study some elementary properties of the Sobolev spaces under consideration and their relation to
classical Sobolev spaces on the interval $(0,\pi)$. We also discuss boundedness of the Poisson-Jacobi
integral maximal operator on some of these spaces.

Throughout the paper we use a standard notation with all symbols referring to the measure space
$((0,\pi),d\theta)$. In particular, we write $L^p$ for $L^p(0,\pi)$ and $\|\cdot\|_p$ for
the associated norm. Further, we set
$$
p(\alpha,\beta) := \begin{cases}
											\infty, & \ab \ge -1/2, \\
											-1/\min(\alpha+1/2,\beta+1/2), & \textrm{otherwise}
										\end{cases}
$$
and
$$
\Psi^{\ab}(\theta) := \Big( \sin\frac{\theta}2\Big)^{\alpha+1/2} \Big(\cos\frac{\theta}2\Big)^{\beta+1/2},
	\qquad \theta \in (0,\pi).
$$

\textbf{Acknowledgment.}
The author would like to express his gratitude to Professor Adam Nowak for 
his constant support during the preparation of this paper.

\section{Preliminaries and preparatory results} \label{sec:tech}

The Jacobi trigonometric functions are defined as
\begin{equation*}
\ph(\theta):=\Psi^{\ab}(\theta)\,\P_n^{\ab}(\theta),\qquad\theta\in(0,\pi),
\end{equation*}
where $\P_n^{\ab}$ are the normalized Jacobi trigonometric polynomials given by
$$
\P_n^{\ab}(\theta) := c_n^{\ab}\, P_n^{\ab}(\cos\theta);
$$
here $c_n^{\ab}$ are normalizing constants, and $P_n^{\ab}$denote 
the classical Jacobi polynomials as defined in Szeg\"o's monograph \cite{Sz}.

Recall that the system $\{\ph:n\ge 0\}$ is an orthonormal basis in $L^2$ consisting of eigenfunctions
of the Jacobi operator,
$$
L_{\ab} \ph = \lambda_n^{\ab} \ph, \quad \textrm{where} \quad
	 \lambda_n^{\ab}:=\big( n+ A_{\ab})^2 \quad \textrm{and}\quad A_{\ab}:= \frac{\alpha+\beta+1}2.
$$
Thus $L_{\ab}$ has a non-negative self-adjoint extension which is natural in this context. It is given by
the spectral series
$$
L_{\ab}f= \sum_{n=0}^{\infty} \lambda_{n}^{\ab} a_n^{\ab}(f)\, \ph
$$
on the domain consisting of those $f \in L^2$ for which this series converges in $L^2$. Here and elsewhere
we denote by $a_n^{\ab}(f)$ the Fourier-Jacobi coefficients of a function $f$ whenever the defining
integrals
$$
a_n^{\ab}(f) := \int_0^{\pi} f(\theta)\, \ph(\theta)\, d\theta
$$
exist. For further reference we also denote
$$
S_{\ab} := \spann\{\ph: n \ge 0 \}.
$$
According to \cite[Lemma 2.3]{Stempak}, $S_{\ab}$ is a dense subspace of $L^p$ provided that 
$1 \le p < p(\ab)$.

The setting related to $L_{\ab}$ was investigated recently in \cite{Nowak&Roncal,NS2,Stempak}.
Its importance comes from the fact that it forms a natural environment for transplantation questions
pertaining to expansions based on Jacobi polynomials, see for instance \cite{CNS,Muckenhoupt}.
Actually, the following result plays a crucial role in our work.
It is essentially a special case of the general weighted multiplier-transplantation theorem due to
Muckenhoupt \cite[Theorem 1.14]{Muckenhoupt}, see \cite[Corollary 17.11]{Muckenhoupt} and also 
\cite[Theorem 2.5]{CNS} together with the related comments on pp{.}\,376--377 therein. Here and elsewhere we use the convention that $\phi_{n}^{\ab} \equiv 0$ if $n<0$.
\begin{lem}[Muckenhoupt]\label{1'}
Let $\alpha,\beta,\gamma,\delta>-1$ and let $d\in\Z$. 
Assume that $g(n)$ is a sequence satisfying for sufficiently large $n$ the smoothness condition
$$
g(n)=\sum_{j=0}^{J-1}c_{j}\,n^{-j}+\mathcal O(n^{-J}),
$$
where $J\ge \alpha+\beta+\gamma+\delta+6$ and $c_j$ are fixed constants.

Then for each $p$ satisfying $p'(\gamma,\delta) < p < p(\alpha,\beta)$ the operator 
$$
f\mapsto\sum_{n=0}^{\infty}g(n)\,a_n^{\ab}(f)\,\phi_{n+d}^{\gamma,\delta}(\theta),\qquad f\in S_{\ab},
$$ 
extends to a bounded operator on $L^p(0,\pi)$. 

\end{lem}

Observe that for $f\in S_{\ab}$ there are only finitely many non-zero terms in the last series.
Moreover, since $S_{\ab}$ is dense in $L^p$ for $p < p(\ab)$, the extension from Lemma \ref{1'} is unique.

The Poisson-Jacobi semigroup $\{\exp(-tL_{\ab}^{1/2})\}_{t\ge 0}$ can be written in $L^2$ by means
of the spectral theorem as
$$
H_t^{\ab}f = \sum_{n=0}^{\infty} \exp\Big(-t\sqrt{\lambda_{n}^{\ab}}\Big)\, a_n^{\ab}(f)\, \ph.
$$
This series converges in fact pointwise and, moreover, may serve as a pointwise definition of $H_t^{\ab}f$,
$t>0$, for more general $f$. In particular, for $f \in L^p$, $p> p'(\ab)$, the coefficients $a_n^{\ab}(f)$
exist and grow polynomially in $n$ (see \cite[Theorem 2.1]{Stempak}) which together with the estimate
(cf.\ \cite[(7.32.2)]{Sz})
\begin{equation} \label{bnd}
|\phi_{n}^{\ab}(\theta)| \le C \, \Psi^{\ab}(\theta)\, (n+1)^{\alpha+\beta+2}, \qquad \theta \in (0,\pi),
 \quad n \ge 0,
\end{equation}
implies pointwise convergence of the series in question (actually, the growth property of $a_n^{\ab}(f)$
is a direct consequence of \eqref{bnd}). Further, $H_t^{\ab}$ has an integral representation 
$$
H_t^{\ab}f(\theta) = \int_0^{\pi} H_t^{\ab}(\theta,\varphi)f(\theta)\, d\theta, \qquad t >0, \quad 
	\theta \in (0,\pi),
$$
valid for $f \in L^p$ with $p$ as before. The integral kernel here is directly related to the Poisson-Jacobi
kernel $\mathcal{H}_t^{\ab}(\theta,\varphi)$ in the context of expansions into $\P_n^{\ab}$ (see
\cite[Section 2]{NS2}),
\begin{equation} \label{rel_p}
H_t^{\ab}(\theta,\varphi) = \Psi^{\ab}(\theta)\Psi^{\ab}(\varphi)\, \mathcal{H}_t^{\ab}(\theta,\varphi).
\end{equation}
Thus sharp estimates of $\mathcal{H}_t^{\ab}(\theta,\varphi)$ obtained in \cite[Theorem 5.2]{NS1} and
\cite[Theorem 6.1]{parameters} imply readily sharp estimates for $H_t^{\ab}(\theta,\varphi)$. Further,
known results on the maximal operator associated with 
$\mathcal{H}_t^{\ab}(\theta,\varphi)$ imply the following.
\begin{prop} \label{prop:max}
Let $\ab > -1$ and let $p \in E(\ab)$. Then the maximal operator
$$
H_{*}^{\ab}f := \sup_{t>0} \big|H_t^{\ab}f\big|
$$
is bounded on $L^p(0,\pi)$.
\end{prop}

\begin{proof}
Let $1<p<\infty$. By \cite[Corollary 2.5]{NS1} and \cite[Corollary 5.2]{parameters}, the maximal
operator
$$
\mathcal{H}_{*}^{\ab}f(\theta) := \sup_{t>0}\bigg| \int_0^{\pi} \mathcal{H}_t^{\ab}(\theta,\varphi)\,
	f(\varphi) \, \Psi^{2\alpha+1/2,2\beta+1/2}(\varphi)\,d\varphi \bigg|
$$
is bounded on $L^p(w\Psi^{2\alpha+1/2,2\beta+1/2})$ for $w \in A_p^{\ab}$, where $A_p^{\ab}$ stands for
the Muckenhoupt class of $A_p$ weights related to the measure 
$\Psi^{2\alpha+1/2,2\beta+1/2}(\theta)\,d\theta$ in $(0,\pi)$, see \cite[Section 1]{NS1}.
Letting $w_{r,s}(\theta) := \Psi^{r-1/2,s-1/2}(\theta)$ be a double-power weight, the condition
$w_{r,s} \in A_p^{\ab}$ is equivalent to saying that $-(2\alpha+2)<r<(2\alpha+2)(p-1)$ and
$-(2\beta+2) < s < (2\beta+2)(p-1)$. The conclusion follows by combining the boundedness of
$\mathcal{H}_{*}^{\ab}$ in double-power weighted $L^p$ with the relation, see \eqref{rel_p},
$H_{*}^{\ab}f = \Psi^{\ab} \mathcal{H}_{*}^{\ab}(\Psi^{-\alpha-1,-\beta-1}f)$.
\end{proof}
By standard arguments, Proposition \ref{prop:max} leads to norm and almost everywhere boundary convergence
of the Poisson-Jacobi semigroup; see the proof of Theorem \ref{thm:conv} below.

Closely related to the Poisson-Jacobi semigroup is the Poisson-Jacobi integral
$$
U_r^{\ab}(f) := \sum_{n=0}^{\infty} r^n\, a_n^{\ab}(f)\, \phi_{n}^{\ab}, \qquad 0<r<1,
$$
and its `spectral' variant
$$
\widetilde{U}_r^{\ab}(f) := 
	\sum_{n=0}^{\infty} r^{|n+A_{\ab}|}\, a_n^{\ab}(f)\, \phi_{n}^{\ab}, \qquad 0<r<1.
$$
The following result extends \cite[Theorem 2.2]{tohoku} in the ultraspherical setting. Some parallel results
obtained recently by different methods can be found in \cite{CU}. 
\begin{thm} \label{thm:conv}
Let $\ab > -1$ and let $p \in E(\ab)$. Then
\begin{enumerate}
\item[(a)] the maximal operators
$$
f \mapsto \sup_{0<r<1}\big|U_r^{\ab}f\big| \qquad \textrm{and} \qquad
f \mapsto \sup_{0<r<1}\big|\widetilde{U}_r^{\ab}f\big| 
$$
are bounded on $L^p(0,\pi)$;
\item[(b)] given any $f\in L^p(0,\pi)$,
$$
U_r^{\ab}f(\theta) \to f(\theta) \qquad \textrm{and} \qquad \widetilde{U}_r^{\ab}f(\theta) \to f(\theta)
	\qquad \textrm{for a.a.}\;\theta\in(0,\pi),
$$
as $r \to 1^{-}$;
\item[(c)] there exists $C>0$ depending only on $\ab$ and $p$ such that
$$
\big\|U_r^{\ab}f\big\|_{L^p(0,\pi)} + \big\|\widetilde{U}_r^{\ab}f\big\|_{L^p(0,\pi)} \le C 
	\|f\|_{L^p(0,\pi)}
$$
for all $0<r<1$ and $f \in L^p(0,\pi)$;
\item[(d)]
for each $f \in L^p(0,\pi)$,
$$
\big\|U_r^{\ab}f - f\big\|_{L^p(0,\pi)} \to 0 \qquad \textrm{and} \qquad 
\big\|\widetilde{U}_r^{\ab}f - f\big\|_{L^p(0,\pi)} \to 0
$$
as $r \to 1^{-}$.
\end{enumerate}
\end{thm}

\begin{proof}
Item (c) is an obvious consequence of (a). Then (b) and (d) follow from (a) and (c) and the density of
$S_{\ab}$ in $L^p$. Thus it remains to prove (a).

In case of $\widetilde{U}_r^{\ab}$, the conclusion follows immediately from Proposition \ref{prop:max}
because $\widetilde{U}_r^{\ab} = H_t^{\ab}$, where $t=-\log r$. To treat the other case, we split the
supremum according to $r \le 1/2$ and $r>1/2$, and denote the resulting maximal operators by
$U_{*,0}^{\ab}$ and $U_{*,1}^{\ab}$, respectively. Then using \eqref{bnd} and H\"older's inequality
we get
\begin{align*}
|U_{*,0}^{\ab}(f)(\theta)| & \le \sum_{n=0}^{\infty} 2^{-n} |a_n^{\ab}(f)| |\phi_n^{\ab}(\theta)| \\
& \le C \|f\|_p \|\Psi^{\ab}\|_{p'} \Psi^{\ab}(\theta) \sum_{n=0}^{\infty} 2^{-n} (n+1)^{2(\alpha+\beta+2)} \\
& \le C \|f\|_p \Psi^{\ab}(\theta).
\end{align*}
This implies the boundedness of $U_{*,0}^{\ab}$. To deal with $U_{*,1}^{\ab}$, we write
$$
U_r^{\ab}(f) = r^{-A_{\ab}} \widetilde{U}_r^{\ab}(f) + \big( 1-r^{|A_{\ab}|-A_{\ab}}\big) \,a_0^{\ab}(f)
\,	\phi_0^{\ab}
$$
and use the $L^p$-boundedness of $\widetilde{U}_r^{\ab}$ and H\"older's inequality. It follows that
$U_{*,1}^{\ab}$ is $L^p$-bounded.
\end{proof}

We remark that a more detailed analysis of the maximal operators of the Poisson-Jacobi integrals and of
boundary convergence of those integrals is possible via the above mentioned sharp estimates for
$H_t^{\ab}(\theta,\varphi)$. Assuming for instance that $\ab \ge -1/2$, one can easily check
by means of \cite[Theorem 6.1]{parameters} that the integral 
kernel $U_r^{\ab}(\theta,\varphi)$ of $U_r^{\ab}$ satisfies
\begin{equation} \label{str}
0 < U_r^{\ab}(\theta,\varphi) \le C \frac{1-r}{(1-r)^2+(\theta-\varphi)^2}, \qquad \theta,\varphi \in (0,\pi),
\quad 0<r<1.
\end{equation}
This extends the estimate for the Poisson-ultraspherical kernel used in \cite{tohoku}, see also
\cite[Lemma 1,\,p.\,27]{MuS}. Moreover, \eqref{str} shows that when $\ab \ge -1/2$, the maximal operators
from Theorem \ref{thm:conv} (a) are controlled by the centered Hardy-Littlewood maximal operator
restricted to $(0,\pi)$. Consequently, (b) of Theorem \ref{thm:conv} holds for $f\in L^1$. Independently,
\eqref{str} gives also (c), and so (d), of Theorem \ref{thm:conv} for $f \in L^1$, still under the assumption
$\ab \ge -1/2$.

Finally, we gather some facts about potential operators associated to $L_{\ab}$. When $\alpha+\beta\neq -1$,
we consider the Riesz type potentials defined for $f\in L^2$ by
$$
L_{\ab}^{-\sigma}f = \sum_{n=0}^{\infty} \big(\lambda_n^{\ab}\big)^{-\sigma} \, a_n^{\ab}(f)\, \phi_n^{\ab}.
$$
In case $\alpha+\beta=-1$ we have $\lambda_0^{\ab}=0$ and thus consider instead the Bessel type potentials
given for $f\in L^2$ by
$$
(\id + L_{\ab})^{-\sigma}f = \sum_{n=0}^{\infty}  
	\big(1+ \lambda_n^{\ab}\big)^{-\sigma} \, a_n^{\ab}(f)\, \phi_n^{\ab}
$$
(notice that this definition makes actually sense for all $\ab > -1$). Clearly, these potentials are
bounded on $L^2$. Further, both $L_{\ab}^{-\sigma}$ and $(\id + L_{\ab})^{-\sigma}$ possess integral
representations that are valid not only for $f \in L^2$, but also for $f \in L^p$ provided that
$p> p'(\ab)$, see \cite{Nowak&Roncal} for more details.

The following result ensures that the definition of the potential spaces $\mathcal{L}_{\ab}^{p,s}$ from
Section \ref{sec:intro} is indeed correct.
\begin{prop} \label{prop:1'}
Let $\ab > -1$ and let $\sigma > 0$. Assume that $p \in E(\ab)$. Then
\begin{enumerate}
\item[(a)] $L_{\ab}^{-\sigma}$ is bounded and one to one on $L^p(0,\pi)$ when $\alpha+\beta \neq -1$;
\item[(b)] $(\id+L_{\ab})^{-\sigma}$ is bounded and one to one on $L^p(0,\pi)$.
\end{enumerate}
\end{prop}
To prove this we need a simple auxiliary property.
\begin{lem} \label{lem:wsp}
Let $\ab$ and $p$ be as in Proposition \ref{prop:1'}. Assume that $f \in L^p(0,\pi)$. 
If $a_n^{\ab}(f)=0$ for all $n \ge 0$, then $f \equiv 0$.
\end{lem}
\begin{proof}
It is enough to observe that the lemma holds for $f \in S_{\ab}$, and then recall that such functions form
a dense subspace in the dual space $(L^p)^* = L^{p'}$.
\end{proof}

\begin{proof}[{Proof of Proposition \ref{prop:1'}}]
The $L^p$-boundedness in (a) and (b) is contained in \cite{Nowak&Roncal}, 
see \cite[Theorem 2.4]{Nowak&Roncal} together with comments on Bessel-Jacobi potentials in 
\cite[Section 1]{Nowak&Roncal}. It can also be obtained
with the aid of Lemma \ref{1'}. To show the remaining assertions we focus on $L_{\ab}^{-\sigma}$; the case
of $(\id + L_{\ab})^{-\sigma}$ is analogous.

As in the proof of \cite[Proposition 1]{betancor}, notice that for $f \in S_{\ab}$
\begin{equation} \label{id}
a_n^{\ab}\big(L_{\ab}^{-\sigma}f\big) = \big(\lambda_n^{\ab}\big)^{-\sigma}\, a_n^{\ab}(f), \qquad n \ge 0.
\end{equation}
Since, by H\"older's inequality and the $L^p$-boundedness of $L_{\ab}^{-\sigma}$, the functionals
$$
f \mapsto a_n^{\ab}\big(L_{\ab}^{-\sigma}f\big) \qquad \textrm{and} \qquad 
f \mapsto a_n^{\ab}(f)
$$
are bounded from $L^p$ to $\mathbb{C}$, and $S_{\ab}$ is dense in $L^p$, we infer that \eqref{id} holds
for $f \in L^p$. Now, if $L_{\ab}^{-\sigma}f \equiv 0$ for some $f \in L^p$, then $a_n^{\ab}(f)=0$ for
all $n \ge 0$ and hence Lemma \ref{lem:wsp} implies $f \equiv 0$. Therefore $L_{\ab}^{-\sigma}$ is one to
one on $L^p$.
\end{proof}

We finish this section by formulating an important consequence of Proposition \ref{prop:1'} and the fact
that $S_{\ab}$ coincides with its images under the action of the potential operators.
\begin{cor} \label{gestosc}
Let $\ab > -1$ and let $s>0$. Assume that $p \in E(\ab)$. Then $S_{\ab}$ is a dense subspace of
$\mathcal{L}_{\ab}^{p,s}$.
\end{cor}

\section{Sobolev spaces defined by variable index derivatives} \label{sec:bet}

The aim of this section is to prove Theorem \ref{thm:sob1}. Thus we let $\mathbb{D}^{(k)}=D^{(k)}$
be the higher order `derivatives' defined by means of the first order `derivatives' related to variable
parameters of type. In what follows we shall generalize the line of reasoning from \cite[Section 3]{betancor}
elaborated in the ultraspherical case. 

To begin with, we look at the action of $D^{(k)}$ and its formal adjoint in $L^2$ on the Jacobi functions.

\begin{lem} \label{12'}
Let $\ab > -1$. Then for any $k,n \ge 0$
\begin{align*}
D^{(k)}\phi_{n}^{\ab} & = (-1)^k \sqrt{(n-k+1)_k \,(n+\alpha+\beta+1)_k} \; \phi_{n-k}^{\alpha+k,\beta+k}, \\
\big(D^{(k)}\big)^* \phi_n^{\alpha+k,\beta+k} & = (-1)^k \sqrt{(n+1)_k \, (n+k+\alpha+\beta+1)_k} \; \phi_{n+k}^{\ab},
\end{align*}
where $(z)_k$ is the Pochhammer symbol, $(z)_k = z(z+1)\ldots (z+k-1)$ when $k\neq 0$ and $(z)_0 = 1$.
\end{lem}

\begin{proof}
To get the first identity it is enough to iterate the formula (see \cite[(4.21.7)]{Sz})
\begin{equation} \label{pochodna}
D_{\ab}\phi_{n}^{\ab}=-\sqrt{n(n+\alpha+\beta+1)}\,\phi_{n-1}^{\alpha+1,\beta+1}.
\end{equation}
To prove the second identity, observe that by \eqref{pochodna} and the relation 
\begin{equation*}
D_{\ab}^*D_{\ab}\,\ph=\big(L_{\ab}-A_{\ab}\big)\,\ph=\big(\lambda_n^{\ab}-\lambda_0^{\ab}\big)\,\ph
\end{equation*}
we have
\begin{equation} \label{15'}
D_{\ab}^*\,\phi_{n-1}^{\alpha+1,\beta+1}=-\sqrt{n(n+\alpha+\beta+1)}\,\phi_n^{\ab},\qquad n\ge1.
\end{equation}
Applying this repeatedly we get the desired conclusion.
\end{proof}

For $k \ge 0$ and $1 \le j \le k$, denote by 
$D^{(k,j)}$ the operator emerging from $D^{(k)}$ by replacing $k$ by $j$, and then $\alpha$
by $\alpha+k-j$ and $\beta$ by $\beta+k-j$, i.e.\
$$
D^{(k,j)} := D_{\alpha+k-1,\beta+k-1} \circ D_{\alpha+k-2,\beta+k-2}\circ \ldots \circ 
	D_{\alpha+k-j,\beta+k-j}. 
$$
Then by the second identity of Lemma \ref{12'} it follows that
\begin{equation} \label{16'}
\big(D^{(k,j)}\big)^* \phi_n^{\alpha+k,\beta+k} =(-1)^j\sqrt{(n+1)_j\, (n+2k-j+\alpha+\beta+1)_j} 
	\;\phi_{n+j}^{\alpha+k-j,\beta+k-j}.
\end{equation}

Next, we state some factorization identities for $D^{(k)}$ and its adjoint. It is elementary to check that
\begin{align*}
D_{\ab}f(\theta)& =\Psi^{\ab}(\theta)\frac{d}{d\theta}\Big(\frac{1}{\Psi^{\ab}(\theta)}f(\theta)\Big), \\
D_{\ab}^*f(\theta)& =-\frac{1}{\Psi^{\ab}(\theta)}\frac{d}{d\theta}\Big(\Psi^{\ab}(\theta)f(\theta)\Big).
\end{align*}
Then with a bit more effort we see that
\begin{align}
D^{(k)}f(\theta) & =\Psi^{\ab}(\theta)(\sin \theta)^k \bigg(\frac{1}{\sin \theta}
 \frac{d}{d\theta}\bigg)^{k}\Big(\frac{1}{\Psi^{\ab}(\theta)}f(\theta)\Big),\nonumber\\
\big(D^{(k)}\big)^*f(\theta)&=\frac{(-1)^{k}}{\Psi^{\ab}(\theta)}\sin \theta \bigg(\frac{1}{\sin \theta}
 \frac{d}{d\theta}\bigg)^{k}\Big((\sin\theta)^{k-1}\Psi^{\ab}(\theta)f(\theta)\Big). \label{fct}
\end{align}
Notice that the last identity implies 
\begin{equation} \label{post*}
\big(D^{(k,j)}\big)^*f(\theta)
=\frac{(-1)^j}{\Psi^{\ab}(\theta)}(\sin\theta)^{-k+j+1} \bigg(\frac{1}{\sin \theta}   
  \frac{d}{d\theta}\bigg)^{j}\Big((\sin\theta)^{k-1}\Psi^{\ab}(\theta)f(\theta)\Big).
\end{equation} 

Our next objective is to demonstrate that $S_{\ab}$ is a dense subspace of the Sobolev spaces.
\begin{prop}\label{p2'}
Let $\alpha,\beta>-1$ and let $m \ge 1$. Assume that $p \in E(\ab)$. 
Then $S_{\ab}$ is a dense subspace of $W_{\ab}^{p,m}$.
\end{prop}

To prove this we will need the following auxiliary technical result.
\begin{lem} \label{lem:32}
Let $\ab, m$ and $p$ be as in Proposition \ref{p2'}. Then for each 
$f\in W_{\ab}^{p,m}$ 
\begin{equation*} 
a_n^{\alpha+k,\beta+k}\big(D^{(k)}f\big)=\int_0^{\pi}f(\theta)\big(D^{(k)}\big)^*\phi_n^{\alpha+k,\beta+k}(\theta)\,d\theta, \qquad 0 \le k \le m, \quad n \ge 0.
\end{equation*}
\end{lem}

\begin{proof}
For $k=0$ there is nothing to prove, so assume that $k\ge1$. 
Choose a sequence of smooth and compactly supported functions $\{\gamma_{l} : l \ge 1\}$ on $(0,\pi)$
satisfying (see the proof of \cite[Proposition 2]{betancor})
\begin{enumerate}
\item[(i)] $\support \gamma_{l}\subset(\frac{1}{2l},\pi-\frac{1}{2l})$, $\gamma_{l}(\theta)=1$ for $\theta\in(\frac{1}{l},\pi-\frac{1}{l})$, $0\le\gamma_l(\theta)\le1$ for $\theta \in (0,\pi)$,
\item[(ii)] for each $r \ge 0$ there exists $C_r>0$ such that
\begin{equation*}
\Big|\frac{d^r}{d\theta^r}\gamma_l(\theta)\Big|\le C_r (\sin\theta)^{-r}, 
	\qquad \theta\in(0,\pi), \quad l \ge 1.
\end{equation*}
\end{enumerate}
By assumption $D^{(k)}f\in L^p$ and so, by H\"older's inequality, the product 
$D^{(k)}f \, \phi_n^{\alpha+k,\beta+k}$ is integrable over $(0,\pi)$. Since $\gamma_{l} \to 1$ pointwise
as $l \to \infty$, the dominated convergence theorem leads to 
\begin{align}
a_n^{\alpha+k,\beta+k}\big(D^{(k)}f\big)&
=\lim\limits_{l \rightarrow
\infty}\int_{0}^{\pi}D^{(k)}f(\theta)\,\gamma_l(\theta)\,\phi_n^{\alpha+k,\beta+k}(\theta)\,d\theta
 \nonumber\\
&=\lim\limits_{l \rightarrow\infty}\int_{0}^{\pi}f(\theta)\,\big(D^{(k)}\big)^*\big[\gamma_l(\theta)\phi_n^{\alpha+k,\beta+k}(\theta)\big]\,d\theta. \label{ocro}
\end{align}

We now analyze the last integral. 
An application of \eqref{fct} and the Leibniz rule yield
\begin{align*}
&\big(D^{(k)}\big)^*\big[\gamma_l(\theta)\phi_n^{\alpha+k,\beta+k}(\theta)\big]\\
&=\frac{(-1)^{k}}{\Psi^{\ab}(\theta)}\sin\theta \bigg(\frac{1}{\sin\theta}\frac{d}{d\theta}\bigg)^{k}\Big((\sin\theta)^{k-1}\Psi^{\ab}(\theta)\gamma_l(\theta)\phi_n^{\alpha+k,\beta+k} (\theta)\Big) \\
&=\frac{(-1)^{k}}{\Psi^{\ab}(\theta)}\sin\theta\sum_{j=0}^{k}\binom{k}{j}\bigg(\frac{1}{\sin\theta}\frac{d}{d\theta}\bigg)^{j}\Big((\sin\theta)^{k-1}\Psi^{\ab}(\theta)\phi_n^{\alpha+k,\beta+k}
(\theta)\Big)\bigg(\frac{1}{\sin \theta}\frac{d}{d\theta}\bigg)^{k-j}\gamma_l(\theta).
\end{align*}
This combined with \eqref{post*} gives
\begin{align*}
& \big(D^{(k)}\big)^*\big[\gamma_l(\theta)\phi_n^{\alpha+k,\beta+k}(\theta)\big]\\
& =\sum_{j=0}^{k}\binom{k}{j}(-1)^{k-j}(\sin\theta)^{k-j} \big(D^{(k,j)}\big)^*
	\phi_n^{\alpha+k,\beta+k}(\theta)
\,\bigg(\frac{1}{\sin \theta}\frac{d}{d\theta}\bigg)^{k-j}\gamma_l(\theta).
\end{align*}
Furthermore, by a straightforward analysis and (ii) we have
\begin{align*}
\bigg|\bigg(\frac{1}{\sin \theta}\frac{d}{d\theta}\bigg)^{r}\gamma_l(\theta)\bigg|\le C_r\sum_{i=1}^r\bigg|\frac{1}{(\sin\theta)^{2r-i}}\frac{d^i}{d\theta^i}\gamma_l(\theta)\bigg|\le C_r\frac{1}{(\sin\theta)^{2r}},\qquad \theta\in(0,\pi),\quad l \ge 1,
\end{align*}
where $r \ge 1$. We conclude that 
\begin{align*}
& \Big|\big(D^{(k)}\big)^*\big[\gamma_l(\theta)\phi_n^{\alpha+k,\beta+k}(\theta)\big]
	-\gamma_l(\theta)\big(D^{(k)}\big)^* \phi_n^{\alpha+k,\beta+k}(\theta)\Big| \\
& \le C_k \bigg| \sum_{j=0}^{k-1} \frac{1}{(\sin\theta)^{k-j}} \big( D^{(k,j)}\big)^* 
	\phi_n^{\alpha+k,\beta+k}(\theta)\bigg|, \qquad \theta \in (0,\pi).
\end{align*}
In view of \eqref{16'}, 
the right-hand side here is controlled by a constant multiple of $\Psi^{\ab}(\theta)$, 
uniformly in $l \ge 1$, and $\Psi^{\ab} \in L^{p'}$ since $p \in E(\ab)$. On the other hand, 
the left-hand side tends to $0$ pointwise, by the choice of $\gamma_{l}$. Thus the dominated convergence
theorem implies
$$
\gr \int_0^{\pi}f(\theta)\big(D^{(k)}\big)^*\big[\gamma_l(\theta)\phi_n^{\alpha+k,\beta+k}(\theta)\big]\,d\theta
=\gr \int_0^{\pi}f(\theta)\,\gamma_l(\theta)\,\big(D^{(k)}\big)^*\phi_n^{\alpha+k,\beta+k}(\theta)\,d\theta,
$$
the integrable majorant being $c\Psi^{\ab}f$. 
Taking into account \eqref{ocro}, this together with another application of the dominated convergence
theorem finishes the proof.
\end{proof}

\begin{proof}[{Proof of Proposition \ref{p2'}}]
We will demonstrate that any function from $W_{\ab}^{p,m}$ can be approximated in the $W_{\ab}^{p,m}$-norm  by partial sums of its Poisson-Jacobi integral. 
The latter functions belong to $S_{\ab}$, which is by Lemma \ref{12'} a subspace of $W^{p,m}_{\ab}$.
For this purpose we need to reveal an interaction between $D^{(k)}$ and $\widetilde{U}_r^{\ab}$.
It turns out that these operators, roughly speaking, almost commute, see \eqref{19'} below. 
This observation is crucial. Then the proof proceeds with the aid of Theorem \ref{thm:conv} (d).

Let $f \in W_{\ab}^{p,m}$ be fixed and let $0 \le k \le m$.
Combining Lemma \ref{lem:32} with the second identity of Lemma \ref{12'} we see that
$$
a_n^{\alpha+k,\beta+k}\big(D^{(k)}f\big)=(-1)^k\sqrt{(n+1)_k\,(n+k+\alpha+\beta+1)_k}\;a_{n+k}^{\ab}(f).
$$
Using this and the first identity of Lemma \ref{12'} we can write
\begin{align} \nonumber
D^{(k)}\widetilde{U}_r^{\ab}(f)(\theta)&=
\sum_{n=k}^{\infty}r^{|n+A_{\ab}|}\,a_n^{\ab}(f)\,(-1)^{k}\,\sqrt{(n-k+1)_k\,(n+\alpha+\beta+1)_k}\,\phi_{n-k}^{\alpha+k,\beta+k}(\theta)\\ \nonumber
&=\sum_{n=k}^{\infty}r^{|n+A_{\ab}|}\,a_{n-k}^{\alpha+k,\beta+k}\big(D^{(k)}f\big)\,\phi_{n-k}^{\alpha+k,\beta+k}(\theta)\\ \nonumber
&=\sum_{n=0}^{\infty}r^{|n+A_{\alpha+k,\beta+k}|}\,a_{n}^{\alpha+k,\beta+k}\big(D^{(k)}f\big)\,\phi_{n}^{\alpha+k,\beta+k}(\theta) \\
& = \widetilde{U}_{r}^{\alpha+k, \beta+k}\big(D^{(k)}f\big)(\theta), \label{19'}
\end{align}
where $0 \le r < 1$ and $\theta \in (0,\pi)$. Exchanging the order of $D^{(k)}$ and the summation in the 
first equality of the above chain is indeed legitimate, as easily verified with the aid of \eqref{bnd}.

Analogous arguments apply to tails of the Poisson-Jacobi integral,
$$
\widetilde{U}_{r,l}^{\ab}(f)(\theta) := 
	\sum_{k=l+1}^{\infty} r^{|n+A_{\ab}|} a_n^{\ab}(f)\, \phi_n^{\ab}(\theta),
$$
producing
$$
D^{(k)} \widetilde{U}_{r,l}^{\ab}(f)(\theta) 
	= \widetilde{U}_{r,l-k}^{\alpha+k,\beta+k}\big(D^{(k)}f\big)(\theta), \qquad l \ge k,
$$
where $r$ and $\theta$ are as before. This identity combined with H\"older's inequality and \eqref{bnd}
leads to the estimates
\begin{align} \nonumber
\big\| D^{(k)} \widetilde{U}_{r,l}^{\ab}(f)\big\|_p & \le \big\|D^{(k)}f\big\|_p \sum_{n=l+1-k}^{\infty}
	r^{|n+A_{\alpha+k,\beta+k}|} \big\|\phi_n^{\alpha+k,\beta+k}\big\|_p\, \big\|\phi_n^{\alpha+k,\beta+k}\big\|_{p'} \\
& \le \big\|D^{(k)}f\big\|_p \sum_{n=l+1-k}^{\infty} r^{|n+A_{\alpha+k,\beta+k}|}
	n^{2(\alpha+\beta+2k+2)}. \label{tk2}
\end{align}

Now, choose an arbitrary $\varepsilon > 0$. By \eqref{19'} and Theorem \ref{thm:conv} (d)
$$
\big\| D^{(k)} \big[ \widetilde{U}_{r_0}^{\ab}(f)- f \big] \big\|_p < \varepsilon, \qquad 0 \le k \le m,
$$
for some $0 < r_0 < 1$. Further, by \eqref{tk2}, there exists $l_0$ depending on $r_0$ such that
$$
\big\| D^{(k)}\widetilde{U}_{r_0,l_0}^{\ab}(f)\big\|_p < \varepsilon, \qquad 0 \le k \le m.
$$
Thus
$$
\big\|  \widetilde{U}_{r_0}^{\ab}(f) - \widetilde{U}_{r_0,l_0}^{\ab}(f) - f
	\big\|_{W_{\ab}^{p,m}} < 2(m+1)\varepsilon.
$$
Since $\widetilde{U}_{r_0}^{\ab}(f) - \widetilde{U}_{r_0,l_0}^{\ab}(f)$ belongs to $S_{\ab}$, 
the conclusion follows.
\end{proof}

We continue by showing $L^p$-boundedness of some variants of higher order Riesz-Jacobi transforms
and certain related operators.
\begin{prop}\label{4'}
Let $\alpha,\beta>-1$ and let $k \ge 0$. Assume that $p\in E(\ab)$. Then the operators
\begin{align*}
R_{\ab}^{k,1}&=D^{(k)}L_{\ab}^{-k/2}, \qquad \alpha+\beta\neq-1,\\
\widetilde{R}_{\ab}^{k,1}&=D^{(k)}(\id+L_{\ab})^{-k/2},
\end{align*}
defined initially on $S_{\ab}$, extend to bounded operators on $L^p(0,\pi)$.
\end{prop}

\begin{proof}
We first focus on $R_{\ab}^{k,1}$. 
Using Lemma \ref{12'} we get
$$
R_{\ab}^{k,1}f=\sum_{n=k}^{\infty}g(n)\,a_n^{\ab}(f)\,\phi_{n-k}^{\alpha+k, \beta+k},\qquad f\in S_{\ab},
$$
where 
$$
g(n)=(-1)^{k}\sqrt{(n-k+1)_k\, (n+\alpha+\beta+1)_k}\;{{|n+A_{\ab}|}^{-k}}
=(-1)^k\sqrt\frac{w(n)}{{(n+A_{\ab})}^{2k}}\,, 
$$
and here $w$ is a polynomial of degree $2k$.

Consider now the function 
$$
h(x)=g\Big(\frac{1}{x}\Big)=(-1)^k\sqrt\frac{x^{2k}w(\frac{1}{x})}{(1+xA_{\ab})^{2k}}.
$$
Here the numerator and the denominator of the fraction under the square root are polynomials, each of them having value 1 at $x=0.$  Thus $h(x)$ is analytic in a neighborhood of $x=0$. In particular, for any fixed $J\ge1$ we have the representation
$$g(n)=h\Big(\frac{1}{n}\Big)=\sum_{j=0}^{J-1}c_j \Big(\frac{1}{n}\Big)^{j}+\mathcal O\bigg(\Big(\frac{1}{n}\Big)^{J}\bigg),$$
provided that $n$ is sufficiently large. Therefore $g$ satisfies the assumptions of Lemma \ref{1'} and the $L^p$-boundedness of $R_{\ab}^{k,1}$ follows.

The case of $\widetilde{R}_{\ab}^{k,1}$ is analogous and is left to the reader.
\end{proof}

The next result states that operators playing the role of conjugates of $R_{\ab}^{k,1}$ and
$\widetilde{R}_{\ab}^{k,1}$ are also bounded on $L^p$.
\begin{prop}\label{5'}
Let $\alpha,\beta>-1$ and let $k \ge 0$. Assume that $p\in E(\ab)$. Then the operators
\begin{align*}
R_{\ab}^{k,2}&=\big(D^{(k)}\big)^*\,L_{\alpha+k,\beta+k}^{-k/2}, \qquad \alpha+\beta \neq -1,\\
\widetilde{R}_{\ab}^{k,2}&=\big(D^{(k)}\big)^*\,(\id+L_{\alpha+k,\beta+k})^{-k/2},
\end{align*}
defined initially on $S_{\alpha+k,\beta+k}$, 
extend to bounded operators on $L^p(0,\pi)$.
\end{prop}
\begin{proof}
Using the second identity of Lemma \ref{12'} we obtain
$$
R_{\ab}^{k,2}f=\sum_{n=0}^{\infty}g(n)\,a_{n}^{\alpha+k,\,\beta+k}(f)\,\phi_{n+k}^{\ab},
	\qquad f \in S_{\alpha+k,\beta+k},
$$
where 
$$
g(n)=\frac{(-1)^k}{(n+k+A_{\ab})^k}\sqrt{(n+1)_k\,(n+k+\alpha+\beta+1)_k}.
$$
As in the proof of Proposition \ref{4'} one verifies that $g(n)$ 
satisfies the assumptions of Lemma \ref{1'} and the conclusion follows.
The treatment of $\widetilde{R}_{\ab}^{k,2}$ relies on the same argument.
\end{proof}
A straightforward computation reveals that for $f\in S_{\ab}$
\begin{align*}
R_{\ab}^{k,2}R_{\ab}^{k,1}f&=\sum_{n=k}^{\infty}(n-k+1)_k\,(n+\alpha+\beta+1)_k
	\,(n+A_{\ab})^{-2k}\,a_n^{\ab}(f)\,\phi_n^{\ab},\\
\widetilde{R}_{\ab}^{k,2}\widetilde{R}_{\ab}^{k,1}f&=
\sum_{n=k}^{\infty}(n-k+1)_k\,(n+\alpha+\beta+1)_k
\,\big(1+(n+A_{\ab})^2\big)^{-k}\,a_n^{\ab}(f)\,\phi_n^{\ab},
\end{align*}
where in the first case we tacitly assume that $\alpha+\beta \neq -1$.
The operators that appear in the proposition below are the inverses of 
$R_{\ab}^{k,2}R_{\ab}^{k,1}$ and $\widetilde{R}_{\ab}^{k,2}\widetilde{R}_{\ab}^{k,1}$, respectively, 
on the subspace 
$$
\big\{f\in S_{\ab}:a_n^{\ab}(f)=0 \;\; \textrm{for} \;\; n \le k-1\big\} \subset S_{\ab}.
$$

\begin{prop}\label{6'}
Let $\alpha,\beta>-1$ and let $k\ge 0$. Assume that $p\in E(\ab)$. Then the operators
\begin{align*}
T_{\ab}^{k}f&=\sum_{n=k}^{\infty}\frac{(n+A_{\ab})^{2k}}{(n-k+1)_k \, (n+\alpha+\beta+1)_k}
	\,a_{n}^{\ab}(f)\,\phi_n^{\ab}, \qquad \alpha+\beta \neq -1,\\
\widetilde{T}_{\ab}^{k}f&=\sum_{n=k}^{\infty}\frac{\big[1+(n+A_{\ab})^2\big]^{k}}{(n-k+1)_k 
\, (n+\alpha+\beta+1)_k}\,a_{n}^{\ab}(f)\,\phi_n^{\ab},
\end{align*}
defined initially on $S_{\ab}$, extend to bounded operators on $L^p(0,\pi)$.
\end{prop}
\begin{proof}
The reasoning is based on a direct application of Lemma \ref{1'}, 
see the proofs of Propositions \ref{4'} and \ref{5'}.
\end{proof} 

Finally, we are in a position to prove Theorem \ref{thm:sob1}.
\begin{proof}[Proof of Theorem \ref{thm:sob1}]
Recall that $S_{\ab}$ is a dense subspace of $W_{\ab}^{p,m}$ (Proposition \ref{p2'}) and of $\Lp$ (Corollary \ref{gestosc}). Moreover, if $f_n\rightarrow f$, either in $W_{\ab}^{p,m}$ or in $\Lp$, then also $f_n\rightarrow f$ in $L^p$. This implication is trivial in case of convergence in $W_{\ab}^{p,m},$ and in the other case it follows by Proposition \ref{prop:1'}. Hence the two spaces have the same elements and to prove that they coincide as Banach spaces it suffices to show that the norms in $W_{\ab}^{p,m}$ and $\Lp$ are equivalent on $S_{\ab}$, i.e.\ there is $C>0$ such that
$$
C^{-1}\|f\|_{W_{\ab}^{p,m}}\le\|f\|_{\Lp}\le C\|f\|_{W_{\ab}^{p,m}},\qquad f\in S_{\ab}.
$$
To proceed, we assume that $\alpha+\beta \neq -1$. The complementary case requires only minor modifications
(including replacements of $R_{\ab}^{m,1}$, $R_{\ab}^{m,2}$ and $T_{\ab}^{m}$ by their tilded counterparts)
and is left to the reader.

Let $f\in S_{\ab}$ and take $g\in S_{\ab}$ such that $f=L_{\ab}^{-m/2}g$. 
We write $g=g_1+g_2$, where $g_1=\sum_{n=0}^{m-1}a_n^{\ab}(g)\phi_n^{\ab}=\sum_{n=0}^{m-1}|n+A_{\ab}|^m a_n^{\ab}(f)\phi_n^{\ab}$. Then observing that $R_{\ab}^{m,1}g_2=R_{\ab}^{m,1}g$ and using Propositions \ref{5'} and \ref{6'} we obtain
\begin{align*}
\|f\|_{\Lp}&=\|g\|_p\le \|g_1\|_{p}+\|g_2\|_{p}\\
&\le\|f\|_p \sum_{n=0}^{m-1} |n+A_{\ab}|^m \|\ph\|_p\,\|\ph\|_{p'}+\big\|T_{\ab}^m R_{\ab}^{m,2} R_{\ab}^{m,1} g\big\|_p\\
&\le C\Big(\|f\|_p+\big\|R_{\ab}^{m,1} g\big\|_p\Big)=C\Big(\|f\|_p+\big\|D^{(m)}f\big\|_p\Big)\le C\|f\|_{W_{\ab}^{p,m}}.
\end{align*}

To prove the reverse estimate we apply Propositions \ref{4'} and \ref{prop:1'} and get
\begin{align*}
\|f\|_{W_{\ab}^{p,m}}&=\sum_{k=0}^m \big\|D^{(k)}f\big\|_p=\sum_{k=0}^m \big\|D^{(k)}L_{\ab}^{-m/2} g\big\|_p=\sum_{k=0}^m \big\|R_{\ab}^{k,1} L_{\ab}^{-(m-k)/2}g\big\|_p\\
&\le C\sum_{k=0}^m \big\|L_{\ab}^{-(m-k)/2}g\big\|_p\le C\|g\|_p=C\|f\|_{\Lp}.
\end{align*}

The proof of Theorem \ref{thm:sob1} is complete.
\end{proof}

\section{Sobolev spaces defined by interlacing derivatives} \label{sec:alternating}

In this section we prove Theorem \ref{thm:sob2}. Thus the higher-order `derivative' under consideration
is $\mathbb{D}^{(k)} = \mathcal{D}^{(k)}$, the operator emerging from interlacing $D_{\ab}$ and its
adjoint. We start with a simple result describing the action of $\mathcal{D}^{(k)}$ on the Jacobi functions.
\begin{lem} \label{lem:inter}
Let $\ab > -1$. Then for any $k,n \ge 0$,
$$
\D^{(k)}\phi_n^{\ab}=(-1)^k\big[n(n+\alpha+\beta+1)\big]^{k/2} 
											\begin{cases} 
													\phi_n^{\ab}, &  k \;\,\textrm{\emph{even}},\\
							 						\phi_{n-1}^{\alpha+1,\beta+1}, & k\;\,\textrm{\emph{odd}}.
											\end{cases}
$$
\end{lem}

\begin{proof}
A direct computation based on \eqref{pochodna} and \eqref{15'}.
\end{proof}

Next, we show that higher-order Riesz-Jacobi transforms defined by means of $\mathcal{D}^{(k)}$ are
bounded on $L^p$.

\begin{prop}\label{4''}
Let $\alpha,\beta>-1$ and let $k \ge 0$. Assume that $p\in E(\ab)$. Then the operators
\begin{align*}
\mathcal R_{\ab}^{k}f&=\mathcal D^{(k)}L_{\ab}^{-k/{2}}f,\qquad \alpha+\beta\neq -1,\\
\widetilde{\mathcal R}_{\ab}^{k}f&=\mathcal D^{(k)}(\id+L_{\ab})^{-k/{2}}f,
\end{align*}
defined initially on $S_{\ab}$, extend to bounded operators on $L^p(0,\pi).$
\end{prop}

\begin{proof}
Consider first the case of $k$ even. According to Lemma \ref{lem:inter}, we have
$$
\mathcal{R}_{\ab}^{k}f=\sum_{n=0}^{\infty}g(n)\, a_n^{\ab}(f)\, \ph,\qquad f \in S_{\ab},
$$
where the multiplier sequence is given by
$$
g(n)= \sqrt{\frac{[n(n+\alpha+\beta+1)]^k}{(n+A_{\ab})^{2k}}}.
$$
As easily verified (see the proof of Proposition \ref{4'}), the sequence $g(n)$ satisfies the 
assumptions of Lemma \ref{1'} and hence the $L^p$-boundedness of $\mathcal R_{\ab}^{k}$ follows.

The case of $k$ odd, as well as the treatment of $\widetilde{\mathcal R}_{\ab}^{k}$, is analogous. 
The conclusion is again a consequence of Lemma \ref{1'}.
\end{proof}

We are now ready to prove Theorem \ref{thm:sob2}.

\begin{proof}[Proof of Theorem \ref{thm:sob2}]
To show the inclusion we assume that $\alpha+\beta \neq -1$. The opposite case requires essentially the
same reasoning and thus is left to the reader.

Let $f\in \mathcal{L}_{\ab}^{p,m}$. By the definition of the potential space, 
there exists $g\in L^p$ such that $f=L^{-m/2}_{\ab}g$ and $\|f\|_{\mathcal{L}_{\ab}^{p,m}}=\|g\|_p$. 
Using Proposition \ref{4''} and then Proposition \ref{prop:1'} we see that, for each $k=0,1,\dots,m$,
\begin{equation*}
\big\|\D^{(k)}f\big\|_p=\big\|\D^{(k)}L_{\ab}^{-{m}/{2}}g\big\|_p
=\big\|\mathcal{R}_{\ab}^{k}L_{\ab}^{-(m-k)/{2}}g\big\|_p
\le C\big\|L_{\ab}^{-(m-k)/{2}}g\big\|_p\le C\|g\|_p
=\|f\|_{\mathcal{L}_{\ab}^{p,m}}.
\end{equation*}
It follows that $\mathcal{L}_{\ab}^{p,m}$ is continuously included in $\mathcal{W}_{\ab}^{p,m}$.

To demonstrate that the reverse inclusion does not hold in general, we give an explicit counterexample. 
For $\ab> -1$, consider the function
$$
f(\theta) = \Psi^{-\alpha,-\beta}(\theta) = \Big(\sin\frac{\theta}2\Big)^{-\alpha+1/2}
	\Big(\cos\frac{\theta}2\Big)^{-\beta+1/2}.
$$
Assume for simplicity that $\alpha \neq 0$ and $\beta \neq 0$.
A direct analysis shows that
\begin{align*}
f(\theta) & \le C\, \theta^{-\alpha+1/2}(\pi-\theta)^{-\beta+1/2},\\
|D_{\ab}f(\theta)| & \le C\, \theta^{-\alpha-1/2}(\pi-\theta)^{-\beta-1/2},\\
|D_{\ab}^* D_{\ab}f(\theta)| & \le C\, f(\theta),\\
|D_{\alpha+1,\beta+1}D_{\ab}f(\theta)| +1 & \ge C\, \theta^{-\alpha-3/2}(\pi-\theta)^{-\beta-3/2},
\end{align*}
where $C>0$ is independent of $\theta \in (0,\pi)$. It is now clear that if 
$p \in E(\alpha,\beta)$, and $\alpha \neq 0$ and
$\beta \neq 0$ are such that $\alpha < 1/p-1/2$ and $\beta < 1/p-1/2$, then $f \in \mathcal{W}_{\ab}^{p,2}$.
On the other hand, $D^{(2)}f \notin L^p$, so in view of Theorem \ref{thm:sob1} one has 
$f \notin \mathcal{L}_{\ab}^{p,2}$.
\end{proof}

Combining Theorem \ref{thm:sob2} with Theorem \ref{thm:sob1} reveals that the Sobolev spaces defined by means
of $\mathbb{D}^{(k)} = {D}^{(k)}$ are contained in those related to 
$\mathbb{D}^{(k)}=\mathcal{D}^{(k)}$, but they do not coincide in general.

\section{Final comments and remarks} \label{sec:fin}

We first point out some natural monotonicity properties of the potential spaces. Analogous facts
are easily seen to be true also for the Sobolev spaces.
\begin{prop}
Let $\ab > -1$. Assume that $p,q \in E(\ab)$. Then
\begin{enumerate}
\item[(a)] if $p \le q$, then $\mathcal{L}_{\ab}^{q,s} \subset \mathcal{L}_{\ab}^{p,s}$ for all $s \ge 0$;
\item[(b)] if $0 \le s \le t$, then $\mathcal{L}_{\ab}^{p,t} \subset \mathcal{L}_{\ab}^{p,s}$ for all $p$.
\end{enumerate}
Moreover, the embeddings in (a) and (b) are continuous.
\end{prop}

\begin{proof}
To get (a) it is enough to use the fact that $\|\cdot\|_p$ is dominated by a constant times $\|\cdot\|_q$
when $p \le q$. Item (b) is a consequence of Proposition \ref{prop:1'}, see the proof of 
\cite[Proposition 6.3]{Graczyk} for the analogous argument in the Laguerre case.
\end{proof}

Next, we comment on the relation between the Sobolev spaces defined in this paper and the classical
Sobolev spaces $W^{p,m}(a,b)$ related to the interval $(a,b)$. The result below shows that there is
only a local connection, and in general $W_{\ab}^{p,m}$ and $\mathcal{W}^{p,m}_{\ab}$ cannot be compared with
$W^{p,m}(0,\pi)$ in terms of inclusion.
\begin{prop}\label{3'}
Let $\ab>-1$, $p \in E(\ab)$ and $m\ge 1$. Assume that $f$ is in $\mathbb{W}_{\ab}^{p,m}$, the Sobolev space
defined either by means of $\mathbb{D}^{(k)} = D^{(k)}$ or by means of $\mathbb{D}^{(k)}=\mathcal{D}^{(k)}$. 
Then
\begin{enumerate}
\item[(a)] $f\in W^{p,m}(a,b)$ whenever $0<a<b<\pi$;
\item[(b)] $f\in W^{p,m}(0,\pi),$ provided that $\support f\subset\subset(0,\pi)$.
\end{enumerate}
Furthermore, none of the inclusions $\mathbb{W}_{\ab}^{p,m} \subset W^{p,m}(0,\pi)$ and
$W^{p,m}(0,\pi)\subset \mathbb{W}_{\ab}^{p,m}$ is true in general. In particular, for 
$(\ab) \neq (-1/2,-1/2)$ there exists $f \in W^{p,m}(0,\pi)$ such that $f \notin \mathbb{W}^{p,m}_{\ab}$,
and for each $\ab$ satisfying $-1/2 \neq \ab \le 1/2-1/p$ there is $f \in \mathbb{W}^{p,m}_{\ab}$
such that $f \notin W^{p,m}(0,\pi)$.
\end{prop}

\begin{proof}
It is not hard to check that $f \in \mathbb{W}_{\ab}^{p,m}$ implies $\frac{d^k}{d\theta^k}f \in L^p(K)$ 
for each compact set $K \subset (0,\pi)$ and $0 \le k \le m$. Thus (a) and (b) follow.

To prove the remaining part, we give explicit counterexamples. Let $f(\theta)\equiv 1$. Clearly,
$f \in W^{p,m}(0,\pi)$. However, a simple computation shows that $D_{\ab}f \notin L^p$ unless
$\alpha=\beta=-1/2$. Thus $f \notin \mathbb{W}_{\ab}^{p,m}$ when $(\ab) \neq (-1/2,-1/2)$.

To disprove the other inclusion, consider $g=\Psi^{\ab}$. Since $g$ is up to a constant factor the Jacobi
function $\phi_0^{\ab}$, we know that $g \in \mathbb{W}_{\ab}^{p,m}$. On the other hand, 
$\frac{d}{d\theta}g \notin L^p$ if $-1/2\neq \alpha \le 1/2-1/p$ or $-1/2\neq \beta \le 1/2-1/p$,
so $g \notin W^{p,m}(0,\pi)$ for the indicated $\alpha$ and $\beta$.
\end{proof}

Finally, we observe that the tools established in this paper allow to generalize the results proved
in \cite[Section 5]{betancor} in the context of ultraspherical expansions. 
\begin{thm} \label{thm:mbs}
Let $\ab > -1$ and assume that $p\in E(\ab)$. Then the maximal operator 
$$
f \mapsto \sup_{0\le r<1}\big|U_r^{\ab}(f)\big|
$$
is bounded on the Sobolev space $\mathbb{W}_{\ab}^{p,1}$ defined by means of 
$\mathbb{D}^{(1)}=D^{(1)}=\mathcal{D}^{(1)}=D_{\ab}$.
\end{thm}

\begin{proof}
We argue in the same way as in the proof of \cite[Theorem 3]{betancor}, replacing the relevant
ultraspherical results by their Jacobi counterparts. More precisely, instead of \cite[Theorem 2.2]{tohoku}
one should use Theorem \ref{thm:conv}. Further, \eqref{19'} should be applied in place of
\cite[(19)]{betancor}, and Proposition \ref{3'} in place of \cite[Proposition 3]{betancor}.
Finally, smoothness of the Poisson-Jacobi integral can be justified directly. Indeed, in view of
\eqref{bnd} and \eqref{pochodna}, the defining series can be differentiated term by term arbitrarily
many times.
\end{proof}

Following the ideas of \cite[Section 5]{betancor}, we also note that \eqref{19'} together with 
\cite[Theorem F]{Gilbert} allow one to conclude boundedness on 
${W}_{\ab}^{p,1}=\mathcal{W}_{\ab}^{p,1}$ of the maximal operator
associated with partial sums of Jacobi expansions, at least when $\ab \ge -1/2$.
\begin{thm}
Let $\ab \ge -1/2$ and let $1<p<\infty$. Then the maximal operator 
$$
f \mapsto \sup_{N\ge 0}\bigg|\sum_{n=0}^{N}a_n^{\ab}(f)\, \ph\bigg|
$$
is bounded on the Sobolev space appearing in Theorem \ref{thm:mbs}.
\end{thm}


\end{document}